\documentclass[12pt,a4paper]{article}      

\usepackage{graphicx}
\usepackage{enumerate}
\usepackage{amssymb, amsmath, amsthm}

\newtheorem{satz}{Theorem}[section]
\newtheorem{lem}[satz]{Lemma} 
\newtheorem{proposition}[satz]{Proposition}

\newtheorem{bsp}[satz]{Example}

\usepackage[T1]{fontenc}
\usepackage{lmodern}

\newcommand{\N}{\ensuremath{{\mathbb N}}}

\newcommand{\R}{\ensuremath{{\mathbb R}}}
\newcommand{\e}{\varepsilon}

\newcommand{\Pv}{\mathbb{P}}
\newcommand{\Ex}{\mathbb{E}}

\newcommand{\norm}[1]{\left \lVert#1 \right\rVert}

\begin{document}

\title{On the Expectation of the Norm of Random Matrices with Non-Identically Distributed Entries}
\author{Stiene Riemer \thanks{Christian-Albrechts-Universit\"at, Mathematisches Seminar, Kiel, Germany,\newline email: \textit{lastname}@math.uni-kiel.de.} \and Carsten Sch\"utt \footnotemark[1]}
\date{\today}
\maketitle

\begin{abstract} Let $X_{i,j}$, $i,j=1,...,n$, be independent, not necessarily identically distributed random variables with finite first moments. We give estimates for the expectation of the norm of the random matrix $(X_{i,j})_{i,j=1}^n$.  We improve a result by R. Latala.
\end{abstract}


\section{Introduction and Notation}\label{intro} 
We study the order of magnitude of the expectation of the largest singular value, i.e. the norm of random matrices with independent entries 
$$
\Ex\left(\norm{(a_{i,j}g_{i,j})_{i,j=1}^n}_{2\rightarrow 2}\right),
$$
where $a_{i,j}\in\mathbb R$, $i,j=1,\dots,n$, $g_{i,j}$, $i,j=1,\dots,n$, are standard Gaussian random variables
and $\|\ \|_{2\to2}$ is the operator norm on $\ell_{2}^{n}$. There are two cases
with a complete answer. Chevet \cite{Che} showed for matrices satisfying
$a_{i,j}=a_{i}b_{j}$ that the expectation is 
proportional to  
$$
\|a\|_{2}\|b\|_{\infty}+\|a\|_{\infty}\|b\|_{2},
$$ 
where $\|a\|_{2}$ denotes the Euclidean of $a=(a_{1},\dots,a_{n})$ and
$\|a\|_{\infty}=\max_{1\leq i\leq n}|a_{i}|$.
\par
For diagonal matrices with diagonal elements
$d_{1},\dots,d_{n}$ we have
that the expectation of the norm is of the order the Orlicz norm
$\|(d_{1},\dots,d_{n})\|_{M}$ where the Orlicz function is given by
$M(s)=\sqrt{\frac{2}{\pi}}\int_{0}^{s}e^{-\frac{1}{2t^{2}}}dt$
\cite{GLSW1}. This Orlicz norm is up to a logarithm of $n$ equal to
to the norm $\max_{1\leq i\leq n}|d_{i}|$.
\par
These two cases are of very different structure and seem to present
essentially what might occur concerning the structure of matrices.
This leads us to conjecture that the expectation for arbitrary matrices is up to a logarithmic
factor equal to
\begin{equation}\label{TrivEst}
\max\limits_{i=1,...,n}\norm{(a_{i,j})_{j=1}^n}_{2}+\max\limits_{j=1,...,n}\norm{(a_{i,j})_{i=1}^n}_{2}.
\end{equation}
\par
Latala \cite{Lat} showed for arbitrary matrices
$$
\Ex\left(\norm{(a_{i,j}g_{i,j})_{i,j=1}^n}_{2\rightarrow 2}\right)\lesssim\max\limits_{i=1,...,n}\norm{(a_{i,j})_{j=1}^n}_{2}+\max\limits_{j=1,...,n}\norm{(a_{i,j})_{i=1}^n}_{2}+\norm{(a_{i,j})_{i,j=1}^n}_4.
$$
Seginer \cite{Seg} showed for any $n\times m$ random matrix $(X_{i,j})_{i,j=1}^{n,m}$
of independent identically distributed random variables
$$
\mathbb E\left\|(X_{i,j})_{i,j=1}^{n,m}\right\|_{2\to2}
\leq c\left(\mathbb E\max_{1\leq i\leq n}\left\|(X_{i,j})_{j=1}^{m}\right\|_{2}
+\mathbb E\max_{1\leq i\leq n}\left\|(X_{i,j})_{j=1}^{m}\right\|_{2}\right).
$$
The largest singular value was first investigated by \cite{Silv, YBK}.
The behavior of the smallest singular value has been determined
in \cite{AGLPT, RV,RV1}.
\vskip 3mm

\begin{satz}\label{thm:gauss_part1gesamt} 
There is a constant $c>0$ such that for all
$a_{i,j}\in\R$, $i,j=1,...,n$, and all independent standard Gaussian random variables $g_{i,j}$, $i,j=1,...,n$, 
 \begin{eqnarray*}
	\left.\begin{split}
	&\Ex\left(\norm{(a_{i,j}g_{i,j})_{i,j=1}^n}_{2\rightarrow 2}\right)  \\
	&\leq c\left(\ln \left(e\frac{\norm{(a_{i,j})_{i,j=1}^n}_1}{\norm{(a_{i,j})_{i,j=1}^n}_{\infty}}\right)\right)^{}\left(\Ex\left(\max\limits_{i=1,...,n}\norm{(a_{i,j}g_{i,j})_{j=1}^n}_2\right)	+\Ex\left(\max\limits_{j=1,...,n}\norm{(a_{i,j}g_{i,j})_{i=1}^n}_2\right)\right).
	\end{split}\right.
	\end{eqnarray*}
\end{satz}
\vskip 3mm

In the same way we prove Theorem \ref{thm:gauss_part1gesamt} we can show the similar formula
\begin{eqnarray*}
	&&\Ex\left(\norm{(a_{i,j}g_{i,j})_{i,j=1}^n}_{2\rightarrow 2}\right)  \\
	&&\leq c\left(\ln\left(e\frac{\norm{(a_{i,j})_{i,j=1}^n}_1}{\norm{(a_{i,j})_{i,j=1}^n}_{\infty}}\right)\right)^{\frac{3}{2}}\left(\max\limits_{i=1,...,n}\norm{(a_{i,j})_{j=1}^n}_2				+\max\limits_{j=1,...,n}\norm{(a_{i,j})_{i=1}^n}_2\right).
		\end{eqnarray*}
This inequality is generalized to arbitrary random variables as in \cite{Lat}.
\vskip 3mm

\begin{satz}\label{thm:allgemein_1} Let $X_{i,j}$, $i,j=1,...,n$, be independent, mean zero random variables. Then
$$
E\left(\norm{(X_{i,j})_{i,j=1}^n}_{2\rightarrow 2}\right)
\leq c\left(\ln(n)\right)^{\frac{3}{2}}\left(\Ex\max\limits_{i=1,...,n}\norm{(X_{i,j})_{j=1}^n}_2+\Ex\max\limits_{j=1,...,n}\norm{(X_{i,j})_{i=1}^n}_2\right).$$
\end{satz}
Since
$\Ex\left(\norm{(a_{i,j}g_{i,j})_{i,j=1}^n}_{2\rightarrow 2}\right)$
is up to a logarithmic factor equal to (\ref{TrivEst}) we investigate better estimate
from below.
On the other hand, 
\begin{equation}\label{lowerexpression}
\Ex\left(\max\limits_{i=1,...,n}\norm{(a_{i,j}g_{i,j})_{j=1}^n}_2\right)+\Ex\left(\max\limits_{j=1,...,n}\norm{(a_{i,j}g_{i,j})_{i=1}^n}_2\right)
\end{equation}
is obviously smaller than 
$\Ex\left(\norm{(a_{i,j}g_{i,j})_{i,j=1}^n}_{2\rightarrow 2}\right)$.
We show that the expression (\ref{lowerexpression}) is equivalent 
to the Musielak-Orlicz norm of the 
vector $(1,\dots,1)$, where the Orlicz functions are given through the
coefficients $a_{i,j}$, $i,j=1,\dots,n$. Our formula (Theorem \ref{thm:orlicz_darstellung}) enables us
to estimate from below the expectation of the operator norm in many cases efficiently.
\par
Moreover, we do not know of any matrix where the expectation of the norm
is not of the same order as (\ref{lowerexpression}).
\par
A convex function $M:[0,\infty)\rightarrow[0,\infty)$ with $M(0)=0$ is called an \textit{Orlicz function} \cite{RaRe}. Let $M$ be an Orlicz function and $x\in\R^n$ then the \textit{Orlicz norm} of $x$, $\norm{x}_M$, is defined by
	$$
\norm{x}_M
=\inf\left\{t>0\left|\sum\limits_{i=1}\limits^nM\left(\frac{|x_i|}{t}\right)\leq 1
\right.\right\}.
$$
	We say that two Orlicz functions $M$ and $N$ are equivalent 
	$(M\sim N)$ if there are
	strictly positive constants $c_{1}$ and $c_{2}$ such that for all $s\geq0$ 
	$$
	M(c_{1}s)\leq N(s)\leq M(c_{2}s).
	$$
	If two Orlicz functions are equivalent, so are their norms: For all 
	$x\in\mathbb R^{n}$
	$$
	c_{1}\|x\|_{M}\leq\|x\|_{N}\leq c_{2}\|x\|_{M}.
	$$
\par
	In addition, let $M_i$, $i=1,...,n$, be Orlicz functions and let $x\in\R^n$ then the \textit{Musielak-Orlicz norm} of $x$, $\norm{x}_{(M_i)_i}$, is defined by
	$$
	\norm{x}_{(M_i)_i}
	=\inf\left\{t>0\left|\sum\limits_{i=1}\limits^nM_i\left(\frac{|x_i|}{t}\right)\leq 1\right.\right\}.
	$$

\section{The upper estimate}\label{sec:part1}

In this section we are going to prove the upper estimate.
We require the following known lemma. In a more general form see e.g.
(\cite{Schue2}, Lemma 10).
\vskip 3mm

\begin{lem}\label{lem:netz} Let $x^{(l)}=\frac{1}{\sqrt{l}}(\overbrace{1,...,1}^{l},\overbrace{0,...,0}^{n-l}),$ $l=1,...,n$,
and let $B_T$ be the convex hull of 
$\left(\e_1x^{(l)}_{\pi(1)},...,\e_nx^{(l)}_{\pi(n)}\right),$
where $\e_i=\pm 1$, $i=1,...,n$, and $\pi$ denote permutations of $\{1,...,n\}$. 
Let $\|\ \|_{T}$ be the norm on $\mathbb R^{n}$ whose unit ball is
$B_{T}$.
Then,
for all $x\in\mathbb R^{n}$
$$
\|x\|_{2}\leq\|x\|_{T}\leq  \sqrt{\ln(en)}\|x\|_{2}.
$$
\end{lem}
\vskip 3mm

\begin{proof} Let $x\in\mathbb R^{n}$. Then $x^{*}_{1},\dots,x^{*}_{n}$
denotes the decreasing rearrangement of the numbers
$|x_{1}|,\dots,|x_{n}|$.
Let
 $a_{k}=\sqrt{k}-\sqrt{k-1}$ for $k=1,\dots,n$. Then, for
all $x\in\mathbb R^{n}$
$$
\|x\|_{T}=\sum_{k=1}^{n}x_{k}^{*}(\sqrt{k}-\sqrt{k-1}).
$$
Since $\sqrt{k}-\sqrt{k-1}\leq\frac{1}{\sqrt{k}}$
$$
\|x\|_{T}
\leq\left(\sum_{k=1}^{n}|\sqrt{k}-\sqrt{k-1}|^{2}\right)^{\frac{1}{2}}
\|x\|_{2}
\leq\left(\sum_{k=1}^{n}\frac{1}{k}\right)^{\frac{1}{2}}
\|x\|_{2}
\leq \sqrt{\ln(en)}\|x\|_{2}
$$
\end{proof}
\vskip 3mm

We denote  
$$
S_{T}^{n-1}=\left\{x=(x_1,...,x_n)\in S^{n-1}\left|\exists i=1,...,n \left|\left\{j=1,...,n|x_j=\pm\frac{1}{\sqrt{i}}\right\}\right|=i\right. \right\}.
$$
Then by our previous lemma we have
\begin{equation}\label{rem:netz}
\norm{A}_{2\rightarrow 2}=\sup\limits_{x\in S^{n-1}}\norm{Ax}_2\leq \sqrt{\ln(en)}\sup\limits_{x\in S_{T}^{n-1}}\norm{Ax}_2.
\end{equation}

We use now the concentration of sums of independent gaussian random variables
$X=\sum_{i=1}^{n}g_{i}z_{i}$ in a Banach space
( \cite{Pi}, Theorem 4.7):
 For all $t>0$
\begin{equation}\label{GaussConcPi}
\mathbb P\{|\|X\|-\mathbb E\|X\||\geq t\}
\leq 2\exp(-Kt^{2}/\sigma(X)^{2}),
\end{equation}
where $K=\frac{2}{\pi^{2}}$ and
\begin{equation}\label{GaussConcPi1}
\sigma(X)=\sup_{\|\xi\|=1}\left(\sum_{i=1}^{n}|\xi(z_{i})|^{2}\right)^{\frac{1}{2}}.
\end{equation}
The following lemma is an immediate consequence.
\vskip 3mm

\begin{lem}\label{lem:konzentration} For all $i,j=1,...,n$ let $a_{i,j}\in\R$, let $g_{i,j}$ be independent standard Gaussians, $G=(a_{i,j}g_{i,j})_{i,j=1}^n$ and let $x\in B_2^n$. For all $\beta\geq 1$ and all $x$ with
$\max\limits_{i=1,...,n}\left(\sum\limits_{j=1}\limits^na_{i,j}^2x_j^2\right)>0$ we have
\begin{eqnarray*}
	\left.\begin{split}
	& \Pv\left(\norm{Gx}_2>\beta
\left(\Ex\left(\max\limits_{i=1,...,n}\norm{(a_{i,j}g_{i,j})_{j=1}^n}_2\right)+\Ex\left(\max\limits_{j=1,...,n}\norm{(a_{ij}g_{ij})_{i=1}^n}_2\right)
\right)\right)\\
& \hskip -10mm\leq
2 \exp\left(-K\frac{\left(\beta
\left(\Ex\left(\max\limits_{i=1,...,n}\norm{(a_{i,j}g_{i,j})_{j=1}^n}_2\right)+\Ex\left(\max\limits_{j=1,...,n}\norm{(a_{i,j}g_{i,j})_{i=1}^n}_2\right)\right)-\left(\sum\limits_{i,j=1}\limits^na_{i,j}^2x_j^2\right)^{\frac{1}{2}}
\right)^2}{\max\limits_{i=1,...,n}\left(\sum\limits_{j=1}\limits^na_{i,j}^2x_j^2\right)}\right).
\end{split}\right.
\end{eqnarray*}
where $K$ is the constant from (\ref{GaussConcPi}).
\end{lem}
\vskip 3mm

Please note that
$$
\|Gx\|_{2\to2}\leq\left(\sum\limits_{i,j=1}\limits^na_{ij}^2x_j^2\right)^{\frac{1}{2}}
\hskip 15mm
\sigma(Gx)
=\max_{i=1\dots,n}\left(\sum\limits_{j=1}\limits^na_{ij}^2x_j^2\right)^{\frac{1}{2}}.
$$
\vskip 3mm

\begin{proposition}\label{lem:obere_ab} For all $i,j=1,...,n$ let $a_{i,j}\in\R$, 
let $g_{i,j}$ be independent standard Gaussian random variables and let 
$G=(a_{i,j}g_{i,j})_{i,j=1}^n$. For all $\beta$ with $\beta\geq\sqrt{\frac{\pi}{2}}$ \begin{eqnarray*}
\left.\begin{split}
& \Pv\left(\norm{G}_{2\rightarrow 2}>\beta \ln(n)
\left(\Ex\left(\max\limits_{i=1,...,n}\norm{(a_{i,j}g_{i,j})_{j=1}^n}_2
+\max\limits_{j=1,...,n}\norm{(a_{i,j}g_{i,j})_{i=1}^n}_2\right)\right)\right)\\
& \le2\sum\limits_{l=1}\limits^n
\exp\left(l\ln(2n)-Kl\frac{2\beta^{2}}{\pi }\right),
\end{split}\right.
\end{eqnarray*}
where $C$ is an absolute constant. Furthermore, we get for 
$\beta$ such that $K\frac{2\beta^{2}}{\pi }=3\ln(2n)$ and $K=\frac{2}{\pi^{2}}$
\begin{eqnarray*}
\left.\begin{split}
& \Pv\left(\norm{G}_{2\rightarrow 2}
>\sqrt{\frac{3\pi^{3}}{4}}\ln(en)
\left(\Ex\left(\max\limits_{i=1,...,n}\norm{(a_{i,j}g_{i,j})_{j=1}^n}_2+
\max\limits_{j=1,...,n}\norm{(a_{i,j}g_{i,j})_{i=1}^n}_2\right)\right)\right)\\
&
\leq\frac{1}{n^2}.
\end{split}\right.
\end{eqnarray*}
\end{proposition}

\begin{proof} We shall apply Lemma \ref{lem:konzentration}. We may assume that
$\max\limits_{i=1,...,n}\left(\sum\limits_{j=1}\limits^na_{i,j}^2x_j^2\right)>0$.
By (\ref{rem:netz})  
	$$
	\norm{G}_{2\rightarrow 2}
	\leq  \sqrt{\ln(en)}\sup\limits_{x\in S_{T}^{n-1}}\norm{Gx}_2.$$
Therefore, for $\beta\in\R_{>0}$, we have
\begin{eqnarray*}
	\left.\begin{split}
	& \Pv\left(\norm{G}_{2\rightarrow 2}>\beta  \sqrt{\ln(en)}
\left(\Ex\left(\max\limits_{i=1,...,n}\norm{(a_{i,j}g_{i,j})_{j=1}^n}_2+
\max\limits_{j=1,...,n}\norm{(a_{i,j}g_{i,j})_{i=1}^n}_2\right)\right)\right)\\
& \leq \Pv\left( \sup\limits_{x\in S_{T}^{n-1}}\norm{Gx}_2>\beta
\left(\Ex\left(\max\limits_{i=1,...,n}\norm{(a_{ij}g_{ij})_{j=1}^n}_2+
\max\limits_{j=1,...,n}\norm{(a_{i,j}g_{i,j})_{i=1}^n}_2\right)\right)
\right).
\end{split}\right.
\end{eqnarray*}
For all $l=1,...,n$ let $M_l$ be the set of $x^{(l)}\in S_T^{n-1}$, 
such that $x^{(l)}_j\in\{0,\pm\frac{1}{\sqrt{l}}\}$ for all $j=1,...,n$. Now we apply Lemma \ref{lem:konzentration} and get
\begin{eqnarray*}
\left.\begin{split}
& \Pv\left(\norm{G}_{2\rightarrow 2}>\beta   \sqrt{\ln(en)}
\left(\Ex\left(\max\limits_{i=1,...,n}\norm{(a_{i,j}g_{i,j})_{j=1}^n}_2+
\max\limits_{j=1,...,n}\norm{(a_{i,j}g_{i,j})_{i=1}^n}_2\right)\right)\right)\\
& \leq \Pv\left( \sup\limits_{x\in S_{T}^{n-1}}\norm{Gx}_2>\beta
\left(\Ex\left(\max\limits_{i=1,...,n}\norm{(a_{ij}g_{ij})_{j=1}^n}_2+
\max\limits_{j=1,...,n}\norm{(a_{i,j}g_{i,j})_{i=1}^n}_2\right)\right)
\right)\\
& \leq\sum\limits_{l=1}\limits^n\sum\limits_{x^{(l)}\in M_l}
\Pv\left(\norm{Gx^{(l)}}_2>\beta
\left(\Ex\left(\max\limits_{i=1,...,n}\norm{(a_{ij}g_{ij})_{j=1}^n}_2+
\max\limits_{j=1,...,n}\norm{(a_{i,j}g_{i,j})_{i=1}^n}_2\right)\right)
\right)\\
& \leq2\sum\limits_{l=1}\limits^n \sum\limits_{x^{(l)}\in M_l}\exp\left(-K\cdot l\left(\beta\frac{\left(\Ex\left(\max\limits_{i=1,...,n}\norm{(a_{ij}g_{ij})_{j=1}^n}_2+
\max\limits_{j=1,...,n}\norm{(a_{i,j}g_{i,j})_{i=1}^n}_2\right)\right)}{\max\limits_{i=1,...,n}\left(\sum\limits_{j\in\{k|x_k^{(l)}\neq 0\}}\limits^na_{ij}^2\right)^{\frac{1}{2}}}\right.\right.\\
& \quad\quad\quad\quad\quad\quad\quad\quad\quad\quad\quad\quad\quad\quad\quad\quad\quad\quad\quad\quad\left.\left.-\frac{\frac{1}{\sqrt{l}}\left(\sum\limits_{i=1}\limits^n\sum\limits_{j\in\{k|x_k^{(l)}\neq 0\}}a_{i,j}^2\right)^{\frac{1}{2}}}{\max\limits_{i=1,...,n}\left(\sum\limits_{j\in\{k|x_k^{(l)}\neq 0\}}\limits^na_{i,j}^2\right)^{\frac{1}{2}}}\right)^2\right).
\end{split}\right.
\end{eqnarray*}
By
$$\frac{1}{\sqrt{l}}\left(\sum\limits_{i=1}\limits^n\sum\limits_{j\in\{k|x_k^{(l)}\neq 0\}}a_{ij}^2\right)^{\frac{1}{2}}\leq\max\limits_{j\in\{k|x_k^{(l)}\neq 0\}}\norm{(a_{ij})_{i=1}^n}_2\leq\max\limits_{j=1,...,n}\norm{(a_{ij})_{i=1}^n}_2$$
and
\begin{equation}\label{lem:obere_ab11}
\Ex\max\limits_{j=1,...,n}\norm{(a_{i,j}g_{i,j})_{i=1}^n}_2
\geq \sqrt{\frac{2}{\pi}} \max\limits_{j=1,...,n}\norm{(a_{i,j})_{i=1}^n}_2
\end{equation} 
we have for all $\beta$ with $\beta\geq \sqrt{\frac{\pi}{2}}$
$$
\beta \Ex\max\limits_{j=1,...,n}\norm{(a_{ij}g_{ij})_i}_2-\frac{1}{\sqrt{l}}
\left(\sum\limits_{i=1}\limits^n\sum\limits_{j\in\{k|x_k^{(l)}\neq 0\}}a_{ij}^2
\right)^{\frac{1}{2}}\geq 0.
$$
Thus
\begin{eqnarray*}
\left.\begin{split}
& \Pv\left(\norm{G}_{2\rightarrow 2}>\beta   \sqrt{\ln(en)}
\left(\Ex\left(\max\limits_{i=1,...,n}\norm{(a_{i,j}g_{i,j})_{j=1}^n}_2+
\max\limits_{j=1,...,n}\norm{(a_{i,j}g_{i,j})_{i=1}^n}_2\right)\right)\right)\\
& \leq2\sum\limits_{l=1}\limits^n \sum\limits_{x^{(l)}\in M_l}
\exp\left(-K\cdot l\left(\beta\frac{\Ex\left(\max\limits_{i=1,...,n}\norm{(a_{i,j}g_{i,j})_{j=1}^n}_2\right)}{\max\limits_{i=1,...,n}\left(\sum\limits_{j\in\{k|x_k^{(l)}\neq 0\}}\limits^na_{i,j}^2\right)^{\frac{1}{2}}}
\right)^2\right).
\end{split}\right.
\end{eqnarray*}
Again, by (\ref{lem:obere_ab11}) we have
for all $\beta$ with $\beta\geq \sqrt{\frac{\pi}{2}}$
\begin{eqnarray*}
\left.\begin{split}
& \Pv\left(\norm{G}_{2\rightarrow 2}>\beta  \sqrt{\ln(en)}
\left(\Ex\left(\max\limits_{i=1,...,n}\norm{(a_{i,j}g_{i,j})_{j=1}^n}_2+
\max\limits_{j=1,...,n}\norm{(a_{i,j}g_{i,j})_{i=1}^n}_2\right)\right)\right)\\
& \leq 2\sum\limits_{l=1}\limits^n \sum\limits_{x^{(l)}\in M_l}
\exp\left(-Kl\frac{2\beta^{2}}{\pi }\right)
\leq2\sum\limits_{l=1}\limits^n2^ln^l
\exp\left(-Kl\frac{2\beta^{2}}{\pi }\right)  \\
&=2\sum\limits_{l=1}\limits^n
\exp\left(l\ln(2n)-Kl\frac{2\beta^{2}}{\pi }\right).
\end{split}\right.
\end{eqnarray*}
We choose $\beta$ such that $3\ln(2n)=K\frac{2\beta^{2}}{\pi }$. Then
\begin{eqnarray*}
\left.\begin{split}
& \Pv\left(\norm{G}_{2\rightarrow 2}>\sqrt{3\ln(2n)}\ln(n)
\left(\Ex\left(\max\limits_{i=1,...,n}\norm{(a_{ij}g_{ij})_{j=1}^n}_2+
\max\limits_{j=1,...,n}\norm{(a_{ij}g_{ij})_{i=1}^n}_2\right)\right)\right)\\
& \leq2 \sum\limits_{l=1}\limits^n\exp\left(l\ln(2n)-3l\ln(2n)\right)
=2\sum\limits_{l=1}\limits^n\exp\left(-2l\ln(2n)\right)
=2\sum\limits_{l=1}\limits^n\left(\frac{1}{4n^2}\right)^l\\
&=2\left(\frac{1-\left(\frac{1}{4n^2}\right)^{n+1}}{1-\frac{1}{4n^2}}-1\right)
=2\frac{1-\left(\frac{1}{4n^2}\right)^n}{4n^2-1}
\leq\frac{1}{n^2}.
\end{split}\right.
\end{eqnarray*}
\end{proof}

\vskip 3mm

\begin{proposition}\label{prop:mitn} Let $a_{i,j}\in\R$, $i,j=1,...,n$, and $g_{i,j}$, $i,j=1,...,n$, be independent standard Gaussian random variables, then 
\begin{eqnarray*}
	\left.\begin{split}
	&\Ex\left(\norm{(a_{i,j}g_{i,j})_{i,j=1}^n}_{2\rightarrow 2}\right)  \\
	&\leq\left(1+\sqrt{\frac{3\pi^{3}}{4}}\ln(en)\right)
\left(\Ex\left(\max\limits_{i=1,...,n}\norm{(a_{i,j}g_{i,j})_{j=1}^n}_2+
\max\limits_{j=1,...,n}\norm{(a_{i,j}g_{i,j})_{i=1}^n}_2\right)\right).
	\end{split}\right.
	\end{eqnarray*}
\end{proposition}
\vskip 3mm

\begin{proof}  We divide the estimate of $\Ex\left(\norm{(a_{i,j}g_{i,j})_{i,j=1}^n}_{2\rightarrow 2}\right)$ into two parts.  Let $M$ be set of all points with
	\begin{eqnarray*}
	&&\norm{(a_{i,j}g_{i,j})_{i,j=1}^n}_{2\rightarrow 2}  \\
	&&\leq\sqrt{\frac{3\pi^{3}}{4}}\ln(en)
	\Ex\left(\max\limits_{i=1,...,n}\norm{(a_{i,j}g_{i,j})_{j=1}^n}_2+
	\max\limits_{j=1,...,n}\norm{(a_{i,j}g_{i,j})_{i=1}^n}_2\right).
	\end{eqnarray*}
Clearly,
\begin{eqnarray*}
	\left.\begin{split}
	&\Ex\left(\norm{(a_{i,j}g_{i,j})_{i,j=1}^n}_{2\rightarrow 2}\chi_M\right) \\
	&\leq\sqrt{\frac{3\pi^{3}}{4}}\ln(en)
	\Ex\left(\max\limits_{i=1,...,n}\norm{(a_{i,j}g_{i,j})_{j=1}^n}_2
	+\max\limits_{j=1,...,n}\norm{(a_{i,j}g_{i,j})_{i=1}^n}_2\right).
	\end{split}\right.
\end{eqnarray*}
Furthermore, by Cauchy-Schwarz inequality and Proposition \ref{lem:obere_ab} we get
\begin{eqnarray*}
	\left.\begin{split}
     & \Ex\left(\norm{(a_{i,j}g_{i,j})_{i,j=1}^n}_{2\rightarrow 2}\chi_{M^{c}}\right)
	\leq \sqrt{\Pv\left(M^{c}\right)}\left(\Ex\left(\norm{(a_{i,j}g_{i,j})_{i,j=1}^n}^2_{2\rightarrow 2}\right)\right)^{\frac{1}{2}}
	\\
	& \leq \frac{1}{n}\left(\Ex\left(\norm{(a_{i,j}g_{i,j})_{i,j=1}^n}^2_{2\rightarrow 2}\right)\right)^{\frac{1}{2}}
	\leq \frac{1}{n}\left(\int\sum\limits_{i=1}\limits^n\left(\sum\limits_{j=1}\limits^n|a_{i,j}g_{i,j}|\right)^2d\Pv\right)^{\frac{1}{2}}  \\
	&=\frac{1}{n}\left(\sum_{i,j=1}^{n}|a_{i,j}|^{2}\right)^{\frac{1}{2}}
	\leq\max_{1\leq i\leq n}\left(\sum_{j=1}^{n}|a_{i,j}|^{2}\right)^{\frac{1}{2}}.
	\end{split}\right.  
\end{eqnarray*}
Besides, we obviously have
\begin{eqnarray*}
\begin{split}
&\max\limits_{i=1,...,n}\norm{(a_{ij})_{j=1}^n}_2+\max\limits_{j=1,...,n}\norm{(a_{ij})_{i=1}^n}_2  \\
&\leq \Ex\left(\max\limits_{i=1,...,n}\norm{(a_{ij}g_{ij})_{j=1}^n}_2+\max\limits_{j=1,...,n}\norm{(a_{ij}g_{ij})_{i=1}^n}_2\right).
\end{split}
\end{eqnarray*}
Altogether, this yields
	\begin{eqnarray*}
		\left.\begin{split}
	     & \Ex\left(\norm{(a_{i,j}g_{i,j})_{i,j=1}^n}_{2\rightarrow 2}\chi_{M^{c}}\right)
	     \leq \Ex\left(\max\limits_{i=1,...,n}\norm{(a_{i,j}g_{i,j})_{j=1}^n}_2+\max\limits_{j=1,...,n}\norm{(a_{i,j}g_{i,j})_{i=1}^n}_2\right).
		\end{split}\right.
	\end{eqnarray*}
Summing up, we get
\begin{eqnarray*}
\left.\begin{split}
& \Ex\left(\norm{(a_{i,j}g_{i,j})_{i,j=1}^n}_{2\rightarrow 2}\right)\\
&\leq
\left(1+\sqrt{\frac{3\pi^{3}}{4}}\ln(en)\right)
\Ex\left(\max\limits_{i=1,...,n}\norm{(a_{i,j}g_{i,j})_{j=1}^n}_2+
\max\limits_{j=1,...,n}\norm{(a_{i,j}g_{i,j})_{i=1}^n}_2\right).
\end{split}\right.
\end{eqnarray*}
\end{proof}
\vskip 3mm

\begin{proof} (Theorem \ref{thm:gauss_part1gesamt}) W.l.o.g. we assume $a_{i,j}\leq 1$, $i,j=1,...,n$, and that there is a coordinate
that equals $1$. For all $i,j=1,...,n$ and $k\in\N$ we define
	\begin{equation*}
	  a_{i,j}^k = 
	   \left\{ 
	    \begin{array}{ll}
	                 \frac{1}{2^k} &  \hskip 5mm
	                 \hbox{if}\quad \frac{1}{2^k}<a_{i,j}\leq \frac{1}{2^{k-1}} \\
	                 0 & \hskip 5mm\hbox{else}.
	    \end{array} 
	   \right.
	 \end{equation*}
Let $G=(a_{i,j}g_{i,j})_{i,j=1}^n$ and $G^k=(a_{i,j}^{k}g_{i,j})_{i,j=1}^n$.
We denote by $\phi(k)$ the number of nonzero entries of the matrix
 $(a_{i,j}^k)_{i,j=1}^n$ and we choose $\gamma$ such that $\norm{(a_{i,j})_{i,j=1}^n}_1=2^{\gamma}\norm{(a_{i,j})_{i,j=1}^n}_{\infty}$. Thus, we get $\phi(k)\frac{1}{2^k}=\sum\limits_{i,j=1}^na_{i,j}^k\leq 2^{\gamma}$ and therefore $\phi(k)\leq 2^{k+\gamma}$.
Therefore, the non-zero entries of $G^{k}$ are contained in a submatrix of size
$2^{k+\gamma}\times 2^{k+\gamma}$.
Taking this into account and applying Proposition \ref{prop:mitn} to $G^{k}$ 
\begin{eqnarray*}
\left.\begin{split}
&\Ex\norm{G^k}_{2\rightarrow 2}  \\
&\leq\left(1+\sqrt{\frac{3\pi^{3}}{4}}\ln(e2^{k+\gamma})\right)
\mathbb E\left(\max\limits_{i=1,...,n}\norm{(a^k_{i,j}g_{i,j})_{j=1}^n}_2			+\max\limits_{j=1,...,n}\norm{(a^k_{i,j}g_{i,j})_{i=1}^n}_2\right)  \\
&\leq 140(k+\gamma)
\left(\mathbb E\sum_{i,j=1}^{n}|a_{i,j}^{k}g_{i,j}|^{2}\right)^{\frac{1}{2}}
\leq 140(k+\gamma)2^{\frac{\gamma}{2}-\frac{k}{2}}
\end{split}\right.
\end{eqnarray*}
Therefore,
$$
\sum\limits_{k\geq2\gamma}\limits^{} \Ex\norm{G^k}_{2\rightarrow 2}
\leq140\sum_{k\geq2\gamma}
\frac{k+\gamma}{2^{\frac{k}{4}}}
\leq280\sum_{k=1}^{\infty}
\frac{k}{2^{\frac{k}{4}}}.
$$
Since one of the coordinates of the matrix is $1$
$$
 \Ex\norm{G^1}_{2\rightarrow 2}
 \geq \int_{-\infty}^{\infty}|g|dt
 =\sqrt{\frac{2}{\pi}}.
$$
Therefore, there is a constant $c$ such that
\begin{eqnarray*}
\Ex\norm{G}_{2\rightarrow 2}
\leq2\mathbb E\left\|\sum_{k\leq 2\gamma}G^{k}\right\|_{2\rightarrow 2}
+2\sum_{k>2\gamma}\mathbb \|G^{k}\|_{2\rightarrow 2}
\leq
c\mathbb E\left\|\sum_{k\leq 2\gamma}G^{k}\right\|_{2\rightarrow 2}.
\end{eqnarray*}
The matrix 
$\sum_{k\leq 2\gamma}G^{k}$ has at most
\begin{equation}\label{mainthn1}
\sum_{k\leq 2\gamma}\phi(k)
\leq\sum_{k\leq 2\gamma}2^{\gamma+k}
\leq 2^{3\gamma+1}
\leq\left(\frac{\left\|(a_{i,j})_{i,j=1}^{n}\right\|_{1}}
{\left\|(a_{i,j})_{i,j=1}^{n}\right\|_{\infty}}\right)^{4}
\end{equation}
entries that are different from $0$. Therefore, all nonzero entries
of $\sum_{k\leq 2\gamma}G^{k}$ are contained in a square 
submatrix having less than (\ref{mainthn1}) rows and columns.
We may apply Proposition \ref{prop:mitn} and get with a proper constant
$c$
\begin{eqnarray*}
	\left.\begin{split}
	\Ex\left(\norm{G}_{2\rightarrow 2}\right)  
	\leq&
	c\left(1+\sqrt{\frac{3\pi^{3}}{4}}\ln\left(e\left(\frac{\left\|(a_{i,j})_{i,j=1}^{n}\right\|_{1}}
{\left\|(a_{i,j})_{i,j=1}^{n}\right\|_{\infty}}\right)^{4}\right)\right) \times \\
&\mathbb E
\left(\max\limits_{i=1,...,n}\norm{\left(\sum_{k\leq2\gamma}a_{i,j}^{k}g_{i,j}\right)_{j=1}^n}_2+\max\limits_{j=1,...,n}\norm{\left(\sum_{k\leq2\gamma}a_{i,j}^{k}g_{i,j}\right)_{i=1}^n}_2\right) .
	\end{split}\right.
	\end{eqnarray*}
\end{proof}

\section{The lower estimate}\label{sec:part2}

\begin{satz}\label{thm:orlicz_darstellung} For all $i,j=1,...,n$ let $a_{i,j}\in\R$ and $g_{i,j}$ be independent standard Gaussians. For all $s\in\R_{\geq 0}$ and for all $i=1,...,n$ let
$$
N_i(s)=
\begin{cases}
	s\max\limits_{j=1,\dots,n}|a_{i,j}|e^{-\frac{1}{s^2\max\limits_{j=1,\dots,n}a_{i,j}^2}} &,s<\frac{1}{\norm{(a_{i,j})_{j=1}^n}_2}\\
	\frac{\max\limits_{j=1,\dots,n}|a_{i,j}|}{\norm{(a_{i,j})_{j=1}^n}_2}e^{-\frac{\norm{(a_{i,j})_{j=1}^n}_2^2}{\max\limits_{j=1,\dots,n}a_{i,j}^2}}+\frac{3}{e}\norm{(a_{i,j})_{j=1}^n}_2\left(s-\frac{1}{\norm{(a_{i,j})_{j=1}^n}_2}\right) &,s\geq\frac{1}{\norm{(a_{i,j})_{j=1}^n}_2},
\end{cases}
$$
respectively let for all $s\in\R_{\geq 0}$ and for all $j=1,...,n$
$$\widetilde{N_j}(s)=
\begin{cases}
	s\max\limits_{i=1,\dots,n}|a_{i,j}|e^{-\frac{1}{s^2\max\limits_{i=1,..\dots,n}a_{i,j}^2}} &,s<\frac{1}{\norm{(a_{ij})_{i=1}^n}_2}\\
	\frac{\max\limits_{i=1,..\dots,n}|a_{ij}|}{\norm{(a_{i,j})_{i=1}^n}_2}e^{-\frac{\norm{(a_{i,j})_{i=1}^n}_2^2}{\max\limits_{i=1,\dots,n}a_{i,j}^2}}+\frac{3}{e}\norm{(a_{i,j})_{i=1}^n}_2\left(s-\frac{1}{\norm{(a_{i,j})_{i=1}^n}_2}\right) &,s\geq\frac{1}{\norm{(a_{i,j})_{i=1}^n}_2}.
\end{cases}
$$
Then 
\begin{eqnarray*}
&&c_{1}\left(\norm{(1)_{j=1}^n}_{(N_i)_i}+\norm{(1)_{i=1}^n}_{(\widetilde{N_j})_j} \right)
\\
&&\leq
\Ex\left(\max\limits_{i=1,...,n}\norm{(a_{i,j}g_{i,j})_{j=1}^n}_2+
\max\limits_{j=1,...,n}\norm{(a_{i,j}g_{i,j})_{i=1}^n}_2\right) \\
&&\leq
c_{2}\left(\norm{(1)_{j=1}^n}_{(N_i)_i}+\norm{(1)_{i=1}^n}_{(\widetilde{N_j})_j}
\right),
\end{eqnarray*}
where $c_1$ and $c_2$ are absolute constants.
\end{satz}
\vskip 3mm

The following example is an immediate consequence of Theorem
\ref{thm:orlicz_darstellung}. It covers Toeplitz matrices.

{\footnotesize
\vskip 1mm
\begin{bsp}
Let $A$ be a $n\times n$-matrix such that for all $i,=1\dots,n$ and $k=1,\dots,n$
$$
\left(\sum_{j=1}^{n}|a_{i,j}|^{2}\right)^{\frac{1}{2}}
=\left(\sum_{j=1}^{n}|a_{j,k}|^{2}\right)^{\frac{1}{2}}
$$
and
$$
\max_{1\leq j\leq n}|a_{i,j}|
=\max_{1\leq j\leq n}|a_{j,k}|
$$
Then
\begin{eqnarray*}
&&\Ex\left(\max\limits_{i=1,..\dots,n}\norm{(a_{i,j}g_{i,j})_{j=1}^n}_2+
\max\limits_{j=1,..\dots,n}\norm{(a_{i,j}g_{i,j})_{i=1}^n}_2\right)  \\
&&\sim\max\left\{\left(\sum_{j=1}^{n}|a_{1,j}|^{2}\right)^{\frac{1}{2}},
\sqrt{\ln n}\max_{1\leq j\leq n}|a_{1,j}|\right\}
\end{eqnarray*}
\end{bsp}
}
\noindent
We associate to a random variable $X$ an Orlicz function $M$ by
\begin{equation}\label{OrlRandom}
M(s)=\int_0^s\int\limits_{\frac{1}{t}\leq |X|}|X|d\Pv dt.
\end{equation}
We have
\begin{eqnarray}\label{OrlRandom1}
		\left.\begin{split}
		M(s)
		&=\int\limits_0\limits^s\int\limits_{\frac{1}{t}\leq |X|}|X|d\Pv dt  \\
		&=\int\limits_0\limits^s\left(\frac{1}{t}\Pv\left(|X|\geq\frac{1}{t}\right)+\int\limits_{\frac{1}{t}}\limits^{\infty}\Pv\left(|X|\geq u\right)du\right)dt.
		\end{split}\right.
		\end{eqnarray}
\vskip 3mm

\begin{lem}\label{thm:wichtig_orlicz} There are strictly positive constants
$c_{1}$ and $c_{2}$ such that for all $n\in\mathbb N$, all
independent random variables $X_1,...,X_n$ with finite first moments
and 
for all $x\in\R^n$ 
$$
c_{1}\norm{x}_{(M_i)_i}
\leq \Ex\max\limits_{1\leq i\leq n}|x_iX_i|
\leq c_2\norm{x}_{(M_i)_i},
$$
where $M_1,...,M_n$ are the
 Orlicz functions that are associated to the random variables 
 $X_{1},\dots,X_{n}$ (\ref{OrlRandom}).
\end{lem}
\vskip 3mm

Lemma \ref{thm:wichtig_orlicz} is a generalization of the same result for
identically distributed random variables \cite{GLSW1}. It can be generalized
from the $\ell_{\infty}$-norm to Orlicz norms.
\par
We use the fact \cite{RuW} that for all $s>0$ 
\begin{equation}\label{eqn:integral}
\frac{\sqrt{2\pi}}{(\pi-1)x+\sqrt{x^{2}+2\pi}}e^{-\tfrac{1}{2}x^{2}}
\leq
\sqrt{\tfrac{2}{\pi}}\int_{x}^{\infty}e^{-\tfrac{1}{2}s^{2}}ds
\leq
\sqrt{\tfrac{2}{\pi}}\tfrac{1}{x}e^{-\frac{1}{2}x^{2}} . 
\end{equation} 
\vskip 3mm

\begin{proof}(Theorem \ref{thm:orlicz_darstellung})
	We apply Lemma \ref{thm:wichtig_orlicz} to the random variables
	$$
	X_{i}
	=\left(\sum_{j=1}^{n}|a_{i,j}g_{i,j}|^{2}\right)^{\frac{1}{2}}
	\hskip 20mm i=1,\dots,n.
	$$
Now, it is enough to show that $M_i\sim N_i$ for all $i=1,\dots,n$. We have two  cases.
\par
We consider first $s<\frac{1}{2}\left(\Ex\left(\sum_{j=1}^na_{i,j}^2g_{i,j}^2\right)^{\frac{1}{2}}\right)^{-1}$. 
There are constants $c_{1},c_{2}>0$ such that for all $u$ with $u>2\Ex\left(\sum\limits_{j=1}\limits^na_{i,j}^2g_{i,j}^2\right)^{\frac{1}{2}}$ 
\begin{equation}\label{eqn:w_keit}
\exp\left(-c_{1}^{2}\frac{u^2}{\max\limits_{j=1,\dots,n}a_{i,j}^2}\right)
\leq	\Pv\left(\left(\sum\limits_{j=1}\limits^na_{i,j}^2g_{i,j}^2\right)^{\frac{1}{2}}\geq u\right)\leq  \exp\left(-c_{2}^{2}\frac{u^2}{\max\limits_{j=1,\dots,n}a_{i,j}^2}\right).
\end{equation}
The right-hand side inequality follows from (\ref{GaussConcPi}).
The left-hand side inequality follows from
$$
\sum_{j=1}^{n}a_{i,j}^{2}g_{i,j}^{2}\geq a_{i,1}^{2}g_{i,1}^{2}.
$$
Since $\frac{1}{t}>2\Ex\left(\sum\limits_{j=1}\limits^{n}a_{i,j}^2g_{i,j}^2\right)^{\frac{1}{2}}$, we can apply (\ref{eqn:w_keit}).  Therefore,
\begin{eqnarray*}
\left.\begin{split}
M_i(s)
	& =\int\limits_{0}\limits^s\left\{\frac{1}{t}\Pv
	\left(\left(\sum\limits_{j=1}\limits^{n}a_{i,j}^2g_{i,j}^2\right)^{\frac{1}{2}}
	\geq\frac{1}{t}\right)+\int\limits_{\frac{1}{t}}^{\infty}\Pv\left(\left(\sum\limits_{j=1}\limits^na_{i,j}^2g_{i,j}^2\right)^{\frac{1}{2}}\geq u\right)du\right\}dt\\
	&\leq \int\limits_{0}\limits^s\left\{\frac{1}{t}\exp\left(-\frac{c_{2}^{2}}{t^2\max\limits_{j=1,\dots,n}a_{i,j}^2}\right)+\int\limits_{\frac{1}{t}}^{\infty}\exp\left(-c_{2}^{2}\frac{u^2}{\max\limits_{j=1,\dots,n}a_{i,j}^2}\right)du\right\}dt.
	\end{split}\right.
\end{eqnarray*}
	By (\ref{eqn:integral})
	$$
	M_i(s)\leq \int\limits_0\limits^{s}\left\{\frac{1}{t}\exp\left(-\frac{c_{2}^{2}}{t^2\max\limits_{j=1,\dots,n}a_{i,j}^2}\right)+t\max\limits_{j=1,\dots,n}a_{i,j}^2\exp\left(-\frac{c_{2}^{2}}{t^2\max\limits_{j=1,\dots,n}a_{i,j}^2}\right)\right\}dt.
	$$
Since 
$\frac{1}{t}>2\Ex\left(\sum\limits_{j=1}\limits^{n}a_{i,j}^2g_{i,j}^2\right)^{\frac{1}{2}}
\geq\sqrt{\frac{2}{\pi}}\norm{(a_{i,j})_{j=1}^n}_2$, 
we get
$$
\frac{1}{t}+t\max\limits_{j=1,\dots,n}a_{i,j}^2 
\leq \frac{1}{t}+\sqrt{\frac{\pi}{2}}
\frac{1}{\norm{(a_{i,j})_{j=1}^n}_2}\norm{(a_{i,j})_{j=1}^n}_2^2
\leq\frac{3}{t}.
$$
Thus,
$$
\frac{1}{t}\leq\frac{1}{t}+t\max\limits_{j=1,\dots,n}a_{i,j}^2
\leq\frac{3}{t}.
$$
Altogether, we get
$$
 M_i(s)
	\leq \int\limits_0\limits^{s}\frac{3}{t}\exp\left(-\frac{c_{2}^{2}}{t^2\max\limits_{j=1,\dots,n}a_{i,j}^2}\right)dt
	= \int\limits_{\frac{1}{s}}\limits^{\infty}\frac{3}{u}\exp\left(-\frac{c_{2}^{2}u^{2}}{\max\limits_{j=1,\dots,n}a_{i,j}^2}\right)du.
$$
Passing to a new constant $c_{2}$ and using (\ref{eqn:integral}) we get for all
$s$ with 
$0\leq s<\frac{1}{2}\left(\Ex\left(\sum_{j=1}^na_{i,j}^2g_{i,j}^2\right)^{\frac{1}{2}}\right)^{-1}$
\begin{equation}\label{M1stPart}
M_i(s)
\leq3\int\limits_{\frac{1}{s}}\limits^{\infty}\exp\left(-\frac{c_{2}^{2}u^{2}}{\max\limits_{j=1,\dots,n}a_{i,j}^2}\right)du
\leq \frac{s}{c_{2}}(\max\limits_{j=1,\dots,n}|a_{i,j}|)\exp\left(-\frac{c_{2}^{2}}{s^2\max\limits_{j=1,\dots,n}a_{i,j}^2}\right).
\end{equation}
From this and the definition of $N_{i}$ we get that there is a constant $c$ such that for all $s$ with 
$0\leq s<\frac{1}{2}\left(\Ex\left(\sum_{j=1}^na_{i,j}^2g_{i,j}^2\right)^{\frac{1}{2}}\right)^{-1}$
$$
M_{i}(s)\leq N_{i}(c_{}s).
$$
Indeed, the inequality follows immediately from
(\ref{M1stPart}) provided that $\frac{s}{c_{2}}\leq\frac{1}{2}\left(\sum_{j=1}^n|a_{i,j}|^2\right)^{-\frac{1}{2}}$. If 
$\frac{c_{2}}{2}\left(\sum_{j=1}^n|a_{i,j}|^2\right)^{-\frac{1}{2}}
\leq s\leq
\frac{1}{2}\left(\Ex\left(\sum_{j=1}^na_{i,j}^2g_{i,j}^2\right)^{\frac{1}{2}}\right)^{-1}$
then, by (\ref{M1stPart}) and $\sqrt{\frac{2}{\pi}}\max_{1\leq j\leq n}a_{i,j}\leq\Ex\left(\sum_{j=1}^na_{i,j}^2g_{i,j}^2\right)^{\frac{1}{2}}$,
$$
M_{i}(s)
\leq\frac{2\max\limits_{j=1,\dots,n}a_{i,j}}{c_{2}\Ex\left(\sum_{j=1}^na_{i,j}^2g_{i,j}^2\right)^{\frac{1}{2}}}
\exp\left({-\frac{c_{2}^{2}(\Ex\left(\sum_{j=1}^na_{i,j}^2g_{i,j}^2\right)^{\frac{1}{2}})^{2}}{4\max\limits_{j=1,\dots,n}a_{i,j}^2}}\right)
\leq\frac{\sqrt{2\pi}}{c_{2}}.
$$
Moreover,
$$
\frac{\sqrt{2\pi}}{c_{2}}
\leq
N_{i}\left(\left(\frac{\sqrt{2\pi}}{c_{2}}+1\right)\|(a_{i,j})_{j=1}^{n}\|_{2}^{-1}\right).
$$
Therefore, with a universal constant $c$
the inequality $M_{i}(s)\leq N_{i}(cs)$ also holds for those values of $s$.
The inverse inequality is treated in the same way.
\par
Now we consider $s$ with $s\geq\frac{1}{2}\left(\Ex\left(\sum\limits_{j=1}\limits^{n}a_{i,j}^2g_{ij}^2\right)^{\frac{1}{2}}\right)^{-1}$ and denote
$\alpha=\Ex\left(\sum\limits_{j=1}\limits^{n}a_{ij}^2g_{ij}^2\right)^{\frac{1}{2}}$.
The following holds
\begin{eqnarray*}
	\left.\begin{split}
	& M_i(s)=\int\limits_{0}\limits^s\left\{\frac{1}{t}\Pv\left(\left(\sum\limits_{j=1}\limits^{n}a_{i,j}^2g_{i,j}^2\right)^{\frac{1}{2}}\geq\frac{1}{t}\right)+\int\limits_{\frac{1}{t}}^{\infty}\Pv\left(\left(\sum\limits_{j=1}\limits^na_{i,j}^2g_{i,j}^2\right)^{\frac{1}{2}}\geq u\right)du\right\}dt\\
	& =\int\limits_{0}\limits^{\frac{1}{2\alpha}}\left\{\frac{1}{t}\Pv\left(\left(\sum\limits_{j=1}\limits^{n}a_{i,j}^2g_{i,j}^2\right)^{\frac{1}{2}}\geq\frac{1}{t}\right)+\int\limits_{\frac{1}{t}}^{\infty}\Pv\left(\left(\sum\limits_{j=1}\limits^na_{ij}^2g_{ij}^2\right)^{\frac{1}{2}}\geq u\right)du\right\}dt\\
	&\quad\quad\quad+\int\limits_{\frac{1}{2\alpha}}\limits^s\left\{\frac{1}{t}\Pv\left(\left(\sum\limits_{j=1}\limits^{n}a_{i,j}^2g_{i,j}^2\right)^{\frac{1}{2}}\geq\frac{1}{t}\right)+\int\limits_{\frac{1}{t}}^{\infty}\Pv\left(\left(\sum\limits_{j=1}\limits^na_{i,j}^2g_{i,j}^2\right)^{\frac{1}{2}}\geq u\right)du\right\}dt.
	\end{split}\right.
\end{eqnarray*}
By (\ref{M1stPart}) the first summand is of the order
$$
\frac{\max\limits_{j=1,\dots,n}|a_{i,j}|}
{\Ex\left(\sum\limits_{j=1}\limits^{n}a_{i,j}^2g_{i,j}^2\right)^{\frac{1}{2}}}\exp\left(-\frac{\left(\Ex\left(\sum\limits_{j=1}\limits^{n}a_{i,j}^2g_{i,j}^2\right)^{\frac{1}{2}}\right)^2}{\max\limits_{j=1,\dots,n}a_{i,j}^2}\right).
	$$
	We estimate the second summand. The second summand is less than or equal to
	\begin{eqnarray*}
	\left.\begin{split}
\int\limits_{\frac{1}{2\alpha}}\limits^s\left\{\frac{1}{t}+\Ex\left(\sum\limits_{j=1}\limits^na_{ij}^2g_{ij}^2\right)^{\frac{1}{2}}\right\}dt 
\leq \int\limits_{\frac{1}{2\alpha}}\limits^s3\Ex\left(\sum\limits_{j=1}\limits^na_{i,j}^2g_{i,j}^2\right)^{\frac{1}{2}}dt
	\leq3\Ex\left(\sum\limits_{j=1}\limits^na_{i,j}^2g_{i,j}^2\right)^{\frac{1}{2}}s.	\end{split}\right.
\end{eqnarray*}
Therefore, with a universal constant $c$ we have for all $s$ with
$s\geq\frac{1}{2}\left(\Ex\left(\sum\limits_{j=1}\limits^{n}a_{i,j}^2g_{ij}^2\right)^{\frac{1}{2}}\right)^{-1}$
$$
M_{i}(s)\leq (c-1) s \left(\sum_{j=1}^{n}|a_{i,j}|^{2}\right)^{\frac{1}{2}}
\leq c s \left(\sum_{j=1}^{n}|a_{i,j}|^{2}\right)^{\frac{1}{2}}-1
\leq N_{i}(cs).
$$
Now, we give a lower estimate. By (\ref{OrlRandom}), for all
$s$ with $s\geq\frac{2}{\alpha}$
$$
M_i(s)\geq\int\limits_{\frac{2}{\alpha}}\limits^s\int\limits_{\frac{1}{t}\leq\left(\sum\limits_{j=1}\limits^na_{i,j}^2g_{i,j}^2\right)^{\frac{1}{2}}}\left(\sum\limits_{j=1}\limits^na_{i,j}^2g_{i,j}^2\right)^{\frac{1}{2}}d\Pv dt
\geq\int\limits_{\frac{2}{\alpha}}\limits^s\int\limits_{\frac{\alpha}{2}
	\leq\left(\sum\limits_{j=1}\limits^na_{i,j}^2g_{i,j}^2\right)^{\frac{1}{2}}}\left(\sum\limits_{j=1}\limits^na_{i,j}^2g_{ij}^2\right)^{\frac{1}{2}}d\Pv dt.
$$
By the definition of $\alpha$
\begin{eqnarray*}
	\left.\begin{split}
	M_{i}(s)	&\geq \frac{1}{2}\int\limits_{\frac{2}{\alpha}}\limits^s\Ex\left(\sum\limits_{j=1}\limits^{n}a_{i,j}^2g_{i,j}^2\right)^{\frac{1}{2}}dt\\
	&=\frac{1}{2}\Ex\left(\sum\limits_{j=1}\limits^{n}a_{i,j}^2g_{i,j}^2\right)^{\frac{1}{2}}\left(s-2\left(\Ex\left(\sum\limits_{j=1}\limits^{n}a_{i,j}^2g_{i,j}^2\right)^{\frac{1}{2}}\right)^{-1}\right)\\
	&=\frac{1}{2}\Ex\left(\sum\limits_{j=1}\limits^{n}a_{i,j}^2g_{i,j}^2\right)^{\frac{1}{2}}s-1.
	\end{split}\right.
\end{eqnarray*}
The rest is done as in the case of the upper estimate.
\end{proof}
\vskip 3mm

\bibliographystyle{plain}

\begin{thebibliography}{dudl}



\bibitem{AGLPT}
{\sc R.  Adamczak, O. Gu{\'e}don, A. Litvak,
A. Pajor,  and N. Tomczak-Jaegermann},
{\em Smallest singular value of random matrices with  
independent columns},
Comptes Rendus Math\'ematique. Acad\'emie des Sciences.  
Paris 346 (2008), 853--856


\bibitem{Che}
{\sc S. Chevet},
{\em S\'eries des variables al\'eatoires gaussiennes \'a valeurs 
dans $E\hat\otimes_{\epsilon}F$},
Application aux produits d'espaces de Wiener abstraits,
In 
{\it S\'eminaires sur la G\'eometrie des Espaces de Banach
(1977-1978)}\'Ecole Polytechnique, 1978.


\bibitem{GLSW1}
{\sc Y. Gordon, A. Litvak, C. Sch\"utt and E. Werner}, 
{\em Orlicz Norms of Sequences of Random Variables}, 
 Annals of Probability, 2002, Vol. 30, No. 4, 1833 - 1853
 
 \bibitem{Lat}
 {\sc R. Latala},
 {\em Some estimates of norms of random matrices},
 Proceedings of the American Mathematical Society
 133 (2005), 1273--1282
 
 \bibitem{Pi}
 {\sc G. Pisier},
 {\em The Volume of Convex Bodies and Banach Space Geometry},
 Cambridge University Press, 1989
 
 \bibitem{RV}
 {\sc M. Rudelson und R. Vershynin},
 {\em Smallest singular value of a random rectangular matrix},
 Communications on Pure and Applied Mathematics 62 (2009),
 1707--1739

\bibitem{RV1}
 {\sc M. Rudelson und R. Vershynin},
{\em The least singular value of a random square matrix is  
{$O(n^{-1/2})$}},
Comptes Rendus Math\'ematique. Acad\'emie des Sciences.  
Paris 346 (2008), 893--896


\bibitem{RaRe}
{\sc M.M. Rao und Z.D. Ren},
{\em Theory of Orlicz Spaces},
Marcel Dekker, 1991


\bibitem{RuW}
{\sc M.B. Ruskai and E. Werner},
{\em Study of a class of regularizations of $1/|x|$ using Gaussian
integrals},
SIAM J: Math. Anal. 32 (2000), 435--463


\bibitem{Schue2}
{\sc C. Sch\"utt},
{\em On the Banach-Mazur distance of finite-dimensional symmetric Banach spaces and the hypergeometric distribution},
Studia Mathematica 72 (1982), 109--129


\bibitem{Seg}
{\sc Y. Seginer},
{\em The expected norm of random matrices},
Combinat. Probab. Comput. 9 (2000), 149--166

\bibitem{Silv}
{\sc J. Silverstein},
{\em On the eigenvectors of large dimensional
sample covariance matrices},
Journal of Multivariate Analysis 30 (1989), 1--16

\bibitem{YBK}
{\sc Y.Q. Yin, Z.D. Bai, and P.R. Krishnaiah}
{\em On the limit of the largest eigenvalue of the
large dimensional sample covariance matrix},
Probability Theory and Related Fields 78 (1988), 509-521

\end{thebibliography}

\end{document}